\documentclass[11pt,a4paper,oneside]{amsart}

\newcounter{commentcounter}

\usepackage{amsmath} 
\usepackage{amssymb} 
\usepackage{amsthm} 
\usepackage{stmaryrd} 
\usepackage[english]{babel} 
\usepackage[font=small,justification=centering]{caption} 
\usepackage[nodayofweek]{datetime}
\usepackage[shortlabels]{enumitem} 
\usepackage[T1]{fontenc} 
\usepackage[utf8]{inputenc} 
\usepackage{ifthen} 
\usepackage{mathabx} 
\usepackage{mathtools} 
\usepackage[dvipsnames]{xcolor} 
\usepackage[pdftex,  colorlinks=true, pagebackref=true]{hyperref} 
    \hypersetup{urlcolor=RoyalBlue, linkcolor=RoyalBlue,  citecolor=black}
\usepackage{setspace} 
\usepackage{tikz-cd} 
\usepackage{xfrac} 
\usepackage[capitalize]{cleveref} 
\usepackage[all]{xy}

\renewcommand*{\backref}[1]{}
\renewcommand*{\backrefalt}[4]
{
    \ifcase #1
        No citation in the text.
    \or
        Cited on Page #2.
    \else
        Cited on Pages #2.
    \fi
}

\makeatletter
\@namedef{subjclassname@1991}{Mathematical subject classification 1991}
\@namedef{subjclassname@2000}{Mathematical subject classification 2000}
\@namedef{subjclassname@2010}{Mathematical subject classification 2010}
\@namedef{subjclassname@2020}{Mathematical subject classification 2020}
\makeatother

\newtheorem{thm}{Theorem}[section]

\newtheorem{corollary}[thm]{Corollary}
\newtheorem{prop}[thm]{Proposition}

\newtheorem{prob}{Problem}

\theoremstyle{definition}

\theoremstyle{plain}

    \newtheoremstyle{TheoremNum}
        {8.0pt plus 2.0pt minus 4.0pt}{8.0pt plus 2.0pt minus 4.0pt} 
        {\itshape} 
        {-0.15cm} 
        {\bfseries} 
        {.} 
        { }  
        {\thmname{#1}\thmnote{ \bfseries #3}}
    \theoremstyle{TheoremNum}

\newcommand*{\claimproofname}{My proof}

\DeclareMathOperator{\aster}{\text{\LARGE{\textasteriskcentered}}}

\DeclareMathOperator{\Aut}{\mathrm{Aut}}
\DeclareMathOperator{\Out}{\mathrm{Out}}

\DeclareMathOperator{\sym}{Sym}

\DeclareMathOperator{\Stab}{\mathrm{Stab}}

\newcommand{\calc}{{\mathcal{C}}}

\newcommand{\calf}{{\mathcal{F}}}

\newcommand{\GL}{\mathrm{GL}}

\newcommand{\PSL}{\mathrm{PSL}}

\newcommand{\Sp}{\mathrm{Sp}}

\newcommand{\Mod}{\mathrm{Mod}}

\newcommand{\CAT}{\mathrm{CAT}}

\newcommand{\onto}{\twoheadrightarrow}

\newcommand{\betti}{b^{(2)}}

\def\Z{\mathbb{Z}}

\newcommand{\NN}{\mathbb{N}}
\newcommand{\ZZ}{\mathbb{Z}}

\newcommand{\RR}{\mathbb{R}}
\newcommand{\QQ}{\mathbb{Q}}
\newcommand{\FF}{\mathbb{F}}

\newcommand{\CS}{\mathcal{C}(S)} 
\newcommand{\DV}{\mathcal{D}(V)} 
\newcommand{\NS}{\mathcal{NS}} 

\newcommand{\EG}{\mathrm{EG}} 

\usepackage{tikz}
\usetikzlibrary{arrows,quotes}
\tikzstyle{blackNode}=[fill=black, draw=black, shape=circle]

\title[Problems on handlebody groups]{Problems on handlebody groups}

\author{Naomi Andrew}
\author{Sebastian Hensel}
\author{Sam Hughes}
\author{Richard D. Wade}
\address[N. Andrew, R. Wade]{Mathematical Institute, Andrew Wiles Building, Observatory Quarter, University of Oxford, Oxford OX2 6GG, UK}
\email{naomi.andrew@maths.ox.ac.uk}
\email{wade@maths.ox.ac.uk}
\address[S. Hensel]{Mathematisches Institut der LMU M\"unchen, Theresienstr. 39, 80333 M\"unchen, Germany}
\email{hensel@math.lmu.de}
\address[S. Hughes]{Rheinische Friedrich-Wilhelms-Universit\"at Bonn, Mathematical Institute, Endenicher Allee 60, 53115 Bonn, Germany}
\email{sam.hughes.maths@gmail.com}\email{hughes@math.uni-bonn.de}

\date{\today}

\begin{document}
\begin{abstract}
We survey a number of constructions and open problems related to the handlebody group, with a focus on recent trends in geometric group theory, (co)homological properties, and its relationship to outer automorphism groups of free groups.  We also briefly describe how the \emph{cheap $\alpha$-rebuilding property} of Abert, Bergeron, Fraczyk, and Gaboriau can be applied using the disc complex to deduce results about the homology growth of the handlebody group.
\end{abstract}
\maketitle

\addtocontents{toc}{\protect\setcounter{tocdepth}{0}}

\section*{Introduction} 

A \emph{handlebody group} is the mapping class group $\Mod(V_g)$ of a solid handlebody $V_g$, which is a 3-manifold obtained by either filling in a 2-dimensional surface $S_g$, or alternatively by attaching 3-handles to a solid 3-ball. Handlebodies are fundamental objects in the theory of three-dimensional manifolds, particularly through the theory of Heegaard splittings. 

A key motivation for studying mapping class groups of handlebodies  from the point of view of geometric group theory is that they sit between mapping class groups of surfaces and outer automorphism groups: the handlebody group $\Mod(V_g)$ admits a surjective map to $\Out(F_g)$ and an injective map to $\Mod(S_g)$:
\[\begin{tikzcd}
\Mod(V_g) \arrow[r, tail] \arrow[d, two heads]   &\Mod(S_g)  \\
  \Out(F_g)& 
\end{tikzcd}\]
The horizontal arrow is obtained by restricting the action of a mapping class to the boundary $\partial V_g=S_g$, and the vertical arrow is obtained by looking at the induced action on the fundamental group $\pi_1(V_g)=F_g$. Hence handlebody groups provide a crucial link between two of the most well-studied families in the area. An introduction to handlebody groups from this perspective is given in \cite{Hensel_primer}.

The main purpose of this article is to draw attention to a series of interesting open problems on handlebody groups. We also take the opportunity to show how recent work of Abert, Bergeron, Fraczyk, and Gaboriau allows one to prove vanishing results for $\ell^2$-Betti numbers of handlebody groups (see Theorem~\ref{t:rebuilding}). 

 Other than some notation and definitions given in Section~\ref{s:spaces}, the sections can be read independently, and we invite the reader to choose their favourite topics from the list below:


\setcounter{tocdepth}{1}
\tableofcontents

\subsection*{Acknowledgements}
Thanks to Mladen Bestvina, Benjamin Br\"uck, Marissa Chesser, Ursula Hamenst\"adt, Thomas Ng, Ulrich Oertel, Dan Petersen, Saul Schleimer and Katie Vokes for helpful correspondence during the preparation of this article.

This work has received funding from the European Research Council (grant no. 850930). Andrew and Wade are supported by the Royal Society of Great Britain.  Hughes is supported by a Humboldt fellowship at Universit\"at Bonn.

\addtocontents{toc}{\protect\setcounter{tocdepth}{1}}

\section{The disc complex and handlebody horoballs}\label{s:spaces}


The mapping class group of a surface acts on its associated \emph{curve complex} and \emph{Teichm\"uller space}. As every handlebody group is a subgroup of a mapping class group of a surface it inherits an action on these spaces. However it can be more useful to work with subcomplexes or subspaces which are invariant under the handlebody group,  namely the \emph{disc complex} and \emph{handlebody horoballs}. 

To set up some notation used throughout: we use $V=V_g^{b,p}$ to denote a handlebody of genus $g$ with $b$ marked discs and $p$ marked punctures. Its boundary is a surface $S_g^{b,p}$ with $b$ marked boundary curves and $p$ marked punctures. In most cases, the questions we ask are interesting enough in the closed case and the marked punctures or discs can be dropped.

\subsection{The disc complex I: definition and algebraic properties of stabilisers}
The \emph{curve complex} $\mathcal{C}(S)$ associated to a surface has a vertex for each isotopy class of essential, nonperipheral, simple closed curves on $S$. A set of curves spans a simplex in $\CS$ if and only if their isotopy classes can be realised disjointly.
 
 Recall that we have fixed a handlebody $V$ and an identification $\partial V =S$ of its boundary with $S$. A \emph{meridian} on $S$ is a curve bounding a disc in $V$. Such a disc is determined up to isotopy by its bounding curve. The \emph{disc complex} $\DV$ is the full subcomplex of $\CS$ spanned by meridians. As meridians are preserved by $\Mod(V)$, there is an action of the handlebody group on $\DV$. Simplex stabilisers in $\DV$ can be understood inductively. Let $\sigma$ be a simplex of $\DV$ with vertices given by meridians $c_0,\ldots,c_k$ which span disjoint discs $d_0, \ldots, d_k$ in $V$. Then one can cut along these discs to obtain handlebodies $V_1,\ldots, V_l$ with marked discs (given by the $d_i$). Let $V_i^\circ$ be the handlebody where each disc is replaced by a marked point. The simplex stabiliser in $\Mod(V)$ has a finite-index subgroup $\Stab^0(\sigma)$ which fixes each vertex. This fits into an exact sequence 
 \begin{equation}\label{eqn.stab}
1\to \mathbb{Z}^k \to \Stab^0(\sigma) \to \oplus_{i=1}^l \Mod(V_i^\circ) \to 1,    
 \end{equation} 
where $\mathbb{Z}^k$ is generated by Dehn twists along the meridians. This allows for inductive arguments on complexity of the handlebody. On the other hand, a natural question arises when studying the action of $\Mod(V)$ on $\CS$.

\begin{prob}[Stabilisers of arbitrary curves]\label{prob:stab_in_CS}
Given an arbitrary multicurve $\sigma \in \mathcal{C}(S)$, describe the stabiliser of $\sigma$ in $\Mod(V)$. For example, is it finitely generated?
\end{prob}

There is significant progress in this direction in the work of Oertel, who proved an analogue of the Nielsen--Thurston classification theorem in the setting of the handlebody group \cite{Oertel2002}, although their focus is on the existence of \emph{reducing surfaces} in the handlebody, rather than full stabilisers of incompressible curves. In particular, \cite[Theorem~4.1]{Oertel2002} decomposes a mapping class of a handlebody into a number simpler pieces via cutting along reducing surfaces in $V$. Section~9 of Oertel's paper already contains a number of interesting problems; we will not list them all and instead highlight \cite[Question~9.11]{Oertel2002}:

\begin{prob}[Uniqueness of surface reduction systems (Oertel)] \label{prob:Oertel}
Is the `Nielsen-Thurston' decomposition of a mapping class along reducing surfaces given by \cite[Theorem~4.1]{Oertel2002} unique?
\end{prob}

See also \cite{Carvalho2006} for further work on the behaviour of individual elements. Such a decomposition theorem could be important in relation to the following problem, which is also open:

\begin{prob}[The conjugacy problem]
Is the conjugacy problem solvable in $\Mod(V_g)$?
\end{prob}

While the conjugacy problem is solved in the mapping class group \cite{Hemion1979,Tao2013}, its decidability does not pass to subgroups, not even those of finite index \cite{CollinsMiller}.

\subsection{The disc complex II: large-scale geometry}

The large scale geometry of the disc complex was studied extensively in work of Masur and Schleimer \cite{MasurSchleimer2013} and Hamenst\"adt \cite{HamenstadtAsymptoticI,HamenstadtAsymptoticII, HamenstadtSpottedI, HamenstadtSpottedII}. When $V$ has no marked points or discs:

\begin{itemize}
\item The disc graph is a quasi-convex subspace of the curve graph \cite{MasurMinskyQuasiconvexity}.
\item Despite this, the embedding $\DV \hookrightarrow \CS$ is not a QI embedding. The failure of this to be a QI embedding is controlled by subspaces known as \emph{witnesses}, or \emph{holes} \cite{MasurSchleimer2013}.
\item However, $\DV$ is still Gromov-hyperbolic \cite{MasurSchleimer2013}.
\end{itemize}    

When $V$ has marked points or discs on the boundary, the disc complexes are said to be \emph{spotted}, and Hamenst\"adt showed that when the genus is at least two, disc graphs with one or two spots are no longer hyperbolic. Furthermore, when $g \geq 4$ and there are two spots, the spotted disc graph has infinite asymptotic dimension \cite{HamenstadtSpottedI,HamenstadtAsymptoticII}. To highlight one interesting open problem: while curve complexes are known to be \emph{uniformly hyperbolic}, in the sense that the hyperbolicity constant does not depend on the genus of the surface \cite{Aougab2013, Bowditch2014, CRS2014, HPW2015}, this is not known for the family of disc complexes. The following problem was suggested by Thomas Ng and Saul Schleimer:

\begin{prob}[Uniform hyperbolicity]
Are the disc complexes uniformly hyperbolic?
\end{prob}
 
\subsection{Handlebody horoballs and classifying spaces} 
 
A system of meridians is \emph{simple} if the associated curves cut the surface up into pieces of genus zero, or equivalently the associated disc system cuts the handlebody up into balls.

Restricting to the closed case for simplicity, let $\mathcal{T}_g$ be the Teichm\"uller space of the associated surface. Fix $0<\epsilon< \log(3+\sqrt{8})$. We define $\mathcal{H}^\epsilon_g \subset \mathcal{T}_g$ to be the set of points $\sigma \in \mathcal{T}_g$ such that the set of curves on $\sigma$ of length strictly less than $\epsilon$ form a simple system of meridians on $S$. When $g=1$, Teichm\"uller space is the hyperbolic plane and $\mathcal{H}^\epsilon_g$ is a horoball associated to the unique meridian in $V_g$. In general we will refer to $\mathcal{H}^\epsilon_g$ as a \emph{handlebody horoball}. 

Handlebody horoballs are contractible \cite{HainautPetersen2023, PetersenWade2024}, so $\mathcal{H}^\epsilon_g/\Gamma$ is a classifying space for every torsion-free subgroup $\Gamma < \Mod(V)$. However, one can ask if $\mathcal{H}^\epsilon_g$ is also a \emph{classifying space for proper actions} for $\Mod(V)$. This is equivalent to showing that fixed point sets of finite subgroups of $\Mod(V)$ are contractible. While this is true for $\mathcal{T}_g$ \cite{JiWolpert2010}, it is not clear how the fixed-point sets of finite subgroups of $\Mod(V)$ behave in $\mathcal{H}_g^\epsilon$.

\begin{prob}[Classifying space for proper actions]
    Is $\mathcal{H}_g^\epsilon$ a classifying space for proper actions of $\Mod(V)$, for all sufficiently small $\epsilon$?
\end{prob}

A possible approach to this problem is the \emph{dismantling} technique of Hensel, Osajda and Przytycki, which gives contractibility for the fixed point sets of finite subgroups of $\Mod(V)$ acting on the disc complex \cite[Metatheorem~C]{HOP2014}.

In the same way that moduli space is a complex algebraic variety, one can also ask for algebro-geometric properties of classifying spaces associated to the handlebody group. The following problem was suggested to us by Dan Petersen:

\begin{prob}[Varieties and classifying spaces]
Define a log scheme or log stack whose analytification is a classifying space for the handlebody group.
\end{prob}

Petersen also provided the following motivation:
in \cite{HainautPetersen2023}, Hainaut and Petersen describe the image $\mathcal{H}^\epsilon_g/\Mod(V_g)$ of a handlebody horoball as a kind of deleted neighbourhood in moduli space, and log geometry can be used to understand deleted neighbourhoods. Furthermore, Giansiracusa describes the handlebody modular operad as the derived modular envelope of the framed little disc operad \cite{Giansiracusa2011}, and the framed little disc operad can be described as an operad in log schemes by a construction of Vaintrob \cite{Vaintrob2019}. Hence one should expect an associated log scheme (or stack). See Section~8 of \cite{LogSchemeStuff} for an exposition of log schemes and related topics.

\section{Problems on large-scale geometry}

\subsection{Isoperimetic functions}
The \emph{Dehn} or \emph{first isoperimetric function} of a finitely presented group $G$ measures the `difficulty' of filling loops in any Cayley graph of $G$ with a finite generating set.  The functions do not depend on the choice of finite presentation.

The isoperimetric functions of the Handlebody groups were computed by Hamenst\"adt and Hensel: for $g \geq 3$, the first isoperimetric functions of $\Mod(V_g)$ and $\Mod(V_{g,1})$ are exponential \cite[Theorem~1.1]{HamenstadtHensel2021}. In contrast, the genus $2$ case behaves quite differently: they show that the group $\Mod(V_2)$ admits a proper cocompact action on a $\CAT(0)$ cube complex \cite[Theorem~1.2]{HamenstadtHensel2021}.  In particular, $\Mod(V_2)$ has quadratic Dehn function.

One can also define \emph{higher isoperimetric functions} of a group with finite type Eilenberg--Maclane space.  Roughly speaking, the $n$th isoperimetric function measures the difficulty of filling an $n$-sphere with an $(n+1)$-disc in the universal cover of any finite type $K(G,1)$ space. It remains open what the higher isoperimetric functions of the handlebody group are.

\begin{prob}[Isoperimetry]
    What are the higher isoperimetric functions of the handlebody groups?
\end{prob}

One can also define \emph{homological isoperimetric functions} of a group admitting a normed finite type free resolution $F_\bullet \to \RR$.  Roughly, the $n$th homological isoperimetric function measures the difficulty of filling an $n$-cycle in $F_\bullet$ with an $(n+1)$-chain.  In general the isoperimetric functions and the homological isoperimetric functions are not equal \cite{ABDY2013,brady2020homological}.

\begin{prob}[Homological isoperimetry]
    What are the homological isoperimetric functions of the handlebody groups?
\end{prob}

In another direction, one can also ask if there is a way of isolating the exponential part of the Dehn function of $\Mod(V_g)$. One way of stating this precisely is via the notion of \emph{relative Dehn functions}, as defined by Osin in \cite{Osin2006}. The following problem was suggested by Ursula Hamenst\"adt:

\begin{prob}[Coning off exponential behaviour]
Does there exist a family of subgroups $\mathcal{H}=\{H_\lambda\}_{\lambda \in \Lambda}$ such that $\Mod(V_g)$ is finitely presented relative to $\mathcal{H}$ and so that the relative Dehn function with respect to $\mathcal{H}$ is quadratic?
\end{prob}

Similar questions, which we have not attempted to make precise, can be asked with respect to the existence of \emph{relative hierarchical} structures for $\Mod(V_g)$, which leads us on to our next topic.

\subsection{Hierarchical hyperbolicity and acylindrical actions}
The theory of hierarchically hyperbolic groups (HHGs) was introduced in \cite{BehrstockHagenSisto2017,BehrstockhagenSisto2019} and is now a highly active area of research.  The prototypical examples of HHGs are mapping class groups \emph{ibid.} and many $\CAT(0)$ cubical groups \cite{HagenTsu2020} (see \cite{shepherd2022cubulation} for a quantification of `many'). At its core a \emph{hierarchically hyperbolic structure} on a group $G$ encodes a wealth of geometric information about quasi-flats and hyperbolic directions in Cayley graphs, as well as giving a model geometry with which to study the group.  Chesser completely described which handlebody groups of closed surfaces are HHGs.

\begin{thm}~\cite[Theorem~1.2]{Chesser2022}
    The group $\Mod(V_g)$ is a hierarchically hyperbolic group if and only if $g\leq 2$.
\end{thm}

As HHGs have quadratic Dehn functions, Hamenst\"adt and Hensel's work implies that $\Mod(V_g)$ is not a HHG when $g \geq 3$. Chesser shows  that $\Mod(V_2)$ is a HHG by constructing an explicit \emph{factor system} on the $\CAT(0)$ cube complex defined in \cite[Theorem~1.2]{HamenstadtHensel2021}.





Given the lack of hierarchical hyperbolicity, one can instead turn to \emph{acylindrical hyperbolicity}. Indeed $\Mod(V)$, viewed as a subgroup of the mapping class group $\Mod(S)$, contains many independent pseudo-Anosov elements (such elements can be constructed by either composing high powers of Dehn twists about filling meridians or by lifting fully irreducible elements of $\Out(F_g)$ -- see \cite{Carvalho2006, Hensel_primer}). As the action of $\Mod(S)$ on $\mathcal{C}(S)$ is acylindrical, this implies:

\begin{prop}
The action of $\Mod(V_g)$ on $\mathcal{C}(S_g)$ is non-elementary and acylindrical for $g \geq 2$. Hence $\Mod(V_g)$ is an acylindrically hyperbolic group.
\end{prop}

An element $g$ of an acylindrically hyperbolic group $G$ is a \emph{generalised loxodromic element} if there exists an acylindrical action of $G$ on a Gromov-hyperbolic space $X$ such that $g$ is loxodromic. In particular, the above shows that pseudo-Anosovs in $\Mod(V_g)$ and $\Mod(S_g)$ are generalised loxodromic elements. For $\Mod(S_g)$ there are no more generalised loxodromics \cite[Example~6.5(a)]{Osin2016}, however for $\Mod(V_g)$ we expect there may be more.

\begin{prob}[Understanding acylindrical actions I]
Describe the generalised loxodromic elements of $\Mod(V_g)$.
\end{prob}

A group has a \emph{universal acylindrical action} if it admits an acylindrical action on a hyperbolic space such that every generalised loxodromic element is loxodromic. The mapping class group action on the curve complex is an important example of a \emph{cobounded} universal acylindrical action. Such actions do not exist in general. Abbott showed that Dunwoody's inaccessible group admits no universal acylindrical action \cite{Abbott2016}. Furthermore, Abbott, Balasubramanya and Osin show that there are subgroups of hyperbolic groups which have universal acylindrical actions (coming from the hyperbolic group) yet have no universal acylindrical action which is cobounded \cite[Theorem~7.2]{ABO2019}.


\begin{prob}[Understanding acylindrical actions II] \hfill
\begin{enumerate}
    \item Is there a \emph{cobounded} acylindrical action of $\Mod(V_g)$ for which every pseudo-Anosov element acts loxodromically?
    \item Does $\Mod(V_g)$ admit a universal acylindrical action? If so, can this action be chosen to be cobounded?
\end{enumerate}
\end{prob}

While one might guess that one can use ths disc complex for the above, this is not the case: Chesser  shows that for $g \geq 2$ the action of $\Mod(V_g)$ on $\mathcal{D}(V_g)$ is not acylindrical \cite[Corollary~7.4]{Chesser2022}. One might instead try to start by \emph{electrifying} the disc complex, possibly by adding new cone points coming from the reducing systems found in \cite{Oertel2002}.

In a related direction, Bestvina--Bromberg--Fujiwara use a weakened notion of generalised loxodromics in the from of \emph{WWPD elements} to describe which elements of the mapping class group of a surface have positive \emph{stable commutator length} \cite{BBFscl}. One can also ask to what extent one can apply their ideas in the context of the handlebody group:

\begin{prob}[scl]
Describe the elements of $\Mod(V_g)$ with non-zero stable commutator length.
\end{prob}

This problem is very much related to Problems~\ref{prob:stab_in_CS} and \ref{prob:Oertel}: the lack of a full Nielsen--Thurston decomposition theory in $\Mod(V_g)$ means that more ideas (and we suspect more hyperbolic complexes), are required to use the methods from \cite{BBFscl}.

\subsection{Quasi-isometric rigidity}
It would be remiss of us not to mention the following interesting and likely-to-be incredibly difficult problem concerning the large-scale geometry of $\Mod(V_g)$:
\begin{prob}[QI-rigidity]
    Are handlebody groups quasi-isometrically rigid?
\end{prob}
Note that despite QI rigidity being a difficult open problem, handlebody groups are known to be rigid from the point of view of measure equivalence \cite{HenselHorbez2021} and commensurator rigidity \cite{Hensel2018}.









\section{Lifting problems in \texorpdfstring{$\Out(F_N)$}{Out(FN)}}

\subsection{Growth rates and growth type}

Given a group $G$ and an automorphism $\Phi \in \Out(G)$, we say that the exponential growth rate of $\Phi$ is  \[ \EG(\Phi) = \sup_{w \in G} \left\{\limsup \sqrt[n]{\| \Phi^n(w) \|}\right\}.\]

Here $\|w\|$ represents conjugacy length: the length of a shortest element in the conjugacy class with respect to some fixed, finite generating set; the growth rate is independent of the chosen basis and representative of $\Phi$.

There are two ways to associate a growth rate to an element of $\Mod(V_g)$: through the projection to $F_g$ and the inclusion to $\Mod(S_g)$. In the case of a pseudo-Anosov element, the growth rate is exactly the dilatation; for a reducible element the growth rate is a root of the dilatation associated to some pseudo-Anosov first return map on some complimentary surface, if  there are any. We say that ``the'' growth rate of an element of $\Mod(V_g)$ is that associated to its inclusion into $\Mod(S_g)$.

\begin{prop}
The projection map \[ \pi \colon \Mod(V_g) \to \Out(F_g)\] is decreasing on growth rates.
\end{prop}

\begin{proof}
This is a general property of growth rates under quotient maps: extending a basis of $F_g$ to a generating set of $\pi_1(S_g)$ one sees that $\|w\|$ gives an upper bound for the length of its image in the quotient.
\end{proof}

Given an outer automorphism $\Phi$, we can consider its preimages $\pi^{-1}(\Phi)$.

\begin{prob}[Pulling back growth rates]
    Given $\Phi \in \Out(F_g)$, can \begin{equation} \label{eq:min_EG_of_preimages} \min \{\EG(\Psi) : \Psi \in \pi^{-1}(\Phi)\}\end{equation} be bounded in terms of $\EG(\Phi)$?
\end{prob}

In the case $g=2$, Kim and Seo \cite{KimSeo2023} show that for a fully irreducible element $\min \EG(\Psi) \leq 10 \EG(\Phi)$. 

A feature (or bug, depending on ones perspective) of the study of $\Out(F_g)$ when $g \geq 3$ is that the growth rates of an element and its inverse do not necessarily agree. The ratio $\ln\EG(\Phi)/\ln\EG(\Phi^{-1})$ is bounded above by a constant $C(g)$ \cite{HandelMosher2007} (and the dependence on $g$ is not an artifact of the proof \cite{DKL2014}). In the fully irreducible case an explicit $C(g)$ can be extracted from a proof due to Bestvina and Algom-Kfir \cite{BestvinaAlgomKfir2012}: it is exponential in $g$ (but not known to be optimal).

In $\Mod(S_g)$ there is no such asymmetry and the growth rate of a mapping class and its inverse coincide. This implies that $\max\{\EG(\Phi), \EG(\Phi^{-1})\}$ is a lower bound to \eqref{eq:min_EG_of_preimages}.

For those automorphisms where $\EG(\Phi)=1$, one can further investigate the degree of polynomial growth. Say an automorphism grows \emph{with degree $d$} if for every element $w$ there is a positive constant $B$ so that $\| \Phi^n(w) \| \leq Bn^d$, and for some elements there is a positive constant $A$ so that $An^d \leq \| \Phi^n(w) \|$. In $\Out(F_g)$ there are elements with polynomial growth up to degree $d=g-1$; in $\Mod(S_g)$ there are only periodic and linearly growing elements. 

In the periodic case, a classification of which finite subgroups can be lifted can be obtained by combining results in the literature. McCullough, Miller and Zimmermann \cite{McCulloughMillerZimmermann1989} obtain necessary and sufficient conditions for a finite subgroup $G < \Out(F_g)$ to lift in terms of the existence of a graph of groups $\mathcal{G}$ for the corresponding virtually free subgroup of $\Aut(F_g)$ satisfying certain properties. Furthermore, it follows for instance from Krsti\'c--Lustig--Vogtmann's work \cite[Section 2.4]{KrsticLustigVogtmann2001} that one can list all such $\mathcal{G}$. Taken together this gives a full, if complicated, classification.

The absence of higher degrees of polynomial growth in $\Mod(S_g)$ means that these elements of $\Out(F_g)$ can only be realised by exponentially growing elements of $\Mod(V_g)$, but for linearly growing elements we can ask:

\begin{prob}[Realising linear growth]
    Which linearly growing elements of $\Out(F_g)$ can be realised as linearly growing elements of $\Mod(V_g)$?
\end{prob}

Or more generally:

\begin{prob}[Type preserving lifts]
    Which subgroups of $\Out(F_g)$ have ``type-preserving'' lifts to $\Mod(V_g)$?
\end{prob}

\subsection{Maps from \texorpdfstring{$\Out(F_g)$}{Out(Fg)} and its subgroups}
We can also ask about maps from $\Out(F_g)$ to $\Mod(V_g)$. Hensel shows in \cite{Hensel2017} that there are no ``virtual sections'': if $\Gamma$ is a finite index subgroup of $\Out(F_g)$ then there are no maps $f: \Gamma \to \Mod(V_g)$ so that $\pi \circ f$ is the identity.

Weakening this, we can ask:

\begin{prob}[Quasi-sections]
    Does the quotient to $\Out(F_g)$ admit a quasi-section?
\end{prob}

Here, a quasi-section means a quasi-isometric embedding $f:\Out(F_g) \to \Mod(V_g)$ so that $\pi \circ f$ is at bounded distance from the identity.

Rather than looking at quasi-sections, one can instead consider homomorphisms:

\begin{prob}[Homomorphisms]
    Describe the set of homomorphisms $\Out(F_g) \to \Mod(V_g)$. What are their images?
\end{prob}

A more difficult version of the above problem is to replace $\Out(F_g)$ by a finite-index subgroup. We suspect that every such homomorphism has finite image. 


\section{Problems in the profinite completion}

The profinite completion $\widehat G$ of a finitely generated group $G$ is the inverse limit of finite quotients of $G$, that is \[\widehat G=\varprojlim_{N\trianglelefteqslant_{\text{f.i.}} G} G/N.\]
This is a profinite and hence compact totally disconnected Hausdorff group.  The group $G$ has a natural map $\iota\colon G\to \widehat G$ by $g\mapsto (gN)$, which is injective if and only if $G$ is residually finite.  Remarkably, the main result of \cite{DixonFormanekPolandRibes1982} states that the group $\widehat G$ is determined up to isomorphism by the set \[\calf(G):=\{G/N\ |\ N\trianglelefteqslant_{\text{f.i.}}G  \}/\cong, \]
that is the set of isomorphism types of finite quotients of $G$.  For a detailed picture of this topic see the books of Ribes--Zalesskii \cite{RibesZalesskii2000} and Wilkes \cite{Wilkes2024}.


Recently, there has been considerable interest in low-dimensional topology for studying profinite completions of groups, we provide a small sampling: for free and surface groups see \cite{jaikin2023}, for $3$-manifold groups see \cite{BridsonMcReynoldsReidSpitler2020,Liu2023,WiltonZalesskii2017,WiltonZalesskii2019}, for free-by-cyclic groups see \cite{HughesKudlinska2023}, for mapping class groups see \cite{Boggi2014,BoggiFunar2023curvepants}, and for braid groups see \cite{HoshiMinamideMochizuki2022,MinamideNakamura2022}.

\subsection{Congruence subgroup property}
There is another natural completion that one can take of a handle-body group $\Mod(V)$, namely, its \emph{congruence completion} $\widecheck \Mod (V)$.  This is defined as follows:  Let $H=\pi_1V$ and let $\calc(H)$ denote the set of finite index characteristic subgroups of $H$.  For each $K\in \calc(\pi_1V)$ we obtain a homomorphism $\alpha_K\colon \Mod(V)\to \Out(H/K)$ with image $Q_K$.  We now define $\widecheck \Mod(V)$ to be the inverse limit of the system of finite quotients $\alpha_K\colon \Mod(V)\onto Q_K$.    

\begin{prob}[Congruence subgroup property (CSP)]
    Is $\widehat\Mod (V) = \widecheck \Mod (V)$?
\end{prob}

Note that the CSP is still open for mapping class groups of genus at least 3.  It fails for $\GL_2(\Z)$ when viewed as the mapping class group of the torus.  CSP is known to hold for genus $0$ surfaces with punctures \cite{DiazDonagi1989} (also \cite{McReynolds2012}),  genus $1$ surfaces with punctures \cite{Asada2001} (see also \cite{BuxErshovRapinchuk2011}), genus $2$ surfaces with punctures \cite{Boggi2009}, and for some subgroups of general mapping class groups \cite{klukowski2024congruence}.

We refer to $\widecheck\Mod (V)$ as the \emph{procongruence handlybody group of} $V$.  Largely motivated by Grothendieck's anabelian geometry \cite{Grothendieck1997}, the procongruence mapping class groups have received considerable attention.  There are number of objects used to study the procongruence mapping class group, e.g. the procongruence curve and pants complexes \cite{BoggiFunar2023curvepants,Gropper2023,Wilkes2020pd2}.

\begin{prob}[Procongruence complexes]
    Study the action of $\widecheck\Mod(V)$ on the procongruence curve and pants complexes.
\end{prob}

\begin{prob}[Procongruence Ivanov's Theorem]
    Define a procongruence disc complex and compute its automorphism group.
\end{prob}

\subsection{Goodness}
Let $\mathbf G$ be a topological group.  A $\mathbf G$-module $M$ is called \emph{discrete} if $M$ has the discrete topology and the $\mathbf{G}$-action is continuous.  

Let $G$ be a finitely generated group and suppose that $M$ is a $\widehat G$-module which is finite as an abelian group.  We have that the natural map $\iota\colon G\to \widehat G$ induces a homomorphism on cohomology
\[ \iota_n\colon H^n_c(\widehat G;M)\to H^n(G;M). \]
In \cite{Serre1997}, Serre defines a group to be $n$-\emph{good} if for all $i\leq n$ the map $\iota_i$ is an isomorphism for every finite discrete $\widehat G$-module $M$.  A group is \emph{good in the sense of Serre} if it is $n$-good for all $n$.

\begin{prob}[Goodness]
    Is $\Mod(V)$ good in the sense of Serre?
\end{prob}

It appears that, in general, goodness is open for both mapping class groups and $\Out(F_N)$.  However, it can be deduced for braid groups using work of Lorensen \cite{Lorensen2008}.

\subsection{Profinite rigidity}
Since taking profinite completions is functorial, an inclusion $i\colon H\to G$ of finitely generated residually finite groups induces a map $\widehat i\colon \widehat H\to \widehat G$.  A classical question of Grothendieck asks is $\widehat i$ an isomorphism if and only if $i$ is an isomorphism.  A group $G$ for which this holds is called \emph{Grothendieck rigid}.  Note there are many negative answers to this question, see for example \cite{PlatonovTavgen1986,BridsonGrunewald2004,Bridson2016,corson2023}.  We call any pair $(G,H)$ for which $\widehat i\colon H\to \widehat G$ is isomorphism but $H\not\cong G$, a \emph{Grothendieck pair}.

A finitely generated residually finite group $G$ is \emph{profinitely enveloping} if every finitely generated residually finite group $H$ with $\widehat H\cong \widehat G$ is a Grothendieck pair $(G,H)$.  We say $G$ is \emph{profinitely rigid} if every finitely generated residually finite group $H$ with $\widehat H\cong \widehat G$ is in fact isomorphic to $G$.  For a survey on profinite rigidity the reader is referred to Reid's ICM notes \cite{Reid2018}.

\begin{prob}[Profinite rigidity] Let $V$ be a handlebody.
\begin{enumerate}
    \item Is $\Mod(V)$ Grothendieck rigid?
    \item Is $\Mod(V)$ profinitely enveloping?
    \item Is $\Mod(V)$ profinitely rigid?
\end{enumerate}
\end{prob}

Note that all of these questions are in general open for arithmetic groups, mapping class groups, braid groups, and $\Out(F_N)$.  In fact outside of the world of amenable groups, there are only very few examples of groups that are profinitely rigid \cite{BridsonMcReynoldsReidSpitler2020,BridsonMcReynoldsReidSpitler2021,CheethamWest2022} and they are all commensurable with fundamental groups of $2$ or $3$-dimensional hyperbolic manifolds.

It is known that for the mapping class group $\Mod(S_g)$ the smallest non-abelian finite quotient is $\Sp_{2g}(2)$ by \cite{KielakPierro2020} (see also \cite{Zimmermann2012} for $g=3,4$), for the Braid group \cite{Kolay2023} it is $\sym(n)$ (see \cite{Tan2024} for surface braid groups), and for $\Aut(F_N)$ it is $\PSL_N(2)$ \cite{BaumeisterKielakPierro2019} (see \cite{MecchiaZimmermann2010} for $N=3,4$).

\begin{prob}[Finite quotients]
    What are the smallest non-abelian, non-nilpotent, and non-soluble quotients of $\Mod(V_g)$?   Which finite simple groups are quotients of $\Mod(V_g)$?
\end{prob}

A natural guess for the smallest non-soluble quotient is $\PSL_g(2)$ obtained by the factorisation $\Mod(V_g)\onto \Out(F_g)\onto \PSL_g(2)$.

\section{Problems in (co)homology}


\subsection{A very brief overview}

For background reading on (co)homological properties of the handlebody group, an excellent starting point is the introduction of Hainaut and Petersen's paper \cite{HainautPetersen2023}. Some of the salient features of $H^*(\Mod(V))$ are:

\begin{itemize}
\item $\Mod(V_g)$ has virtual cohomological dimension equal to $4g-5$, which is the same dimension as $\Mod(S_g)$ \cite[Theorem~1.1]{Hirose2003}.
\item $H_1(\Mod(V_g))$ was computed by Wajnryb in \cite{Wajnryb1998}. This had a small error in the $g=2$ case which was corrected in \cite[Lemma~2.4]{IshidaSato2017}.
\item Hatcher and Wahl showed that the maps $\Mod(V_g^1) \to \Mod(V_{g+1}^1)$ given by attaching a handle satisfy homology stability. As a consequence they show that $H_i(\Mod(V_g))$ is independent of $g$ when $g\geq 2i+4$ \cite[Corollary~1.9]{HatcherWahl2010}.  Generalisations with non-trivial coefficient systems have been considered by Randal-Williams and Wahl \cite[Theorem~J]{RandalWilliamsWahl2017}.
\end{itemize}

Hainaut and Petersen show that the quotient $\mathcal{H}_g^\epsilon/\Mod(V_g)$ is an \emph{orbifold classifying space} for the handlebody group \cite{HainautPetersen2023}. They combine this with work of Chan--Galatius--Payne \cite{ChanGalatiusPayne2021} to show that the codimension one Betti numbers of $\Mod(V_g)$  grow at least exponentially in $g$ \cite[Corollary~1.6]{HainautPetersen2023}.

\subsection{Problems related to the dualising module}

A truncated version of $\mathcal{H}_g^\epsilon$ was used in \cite{PetersenWade2024} to show that $\Mod(V)$ is a \emph{virtual duality group}, in the sense of Bieri and Eckmann \cite{BieriEckmann1973}. In particular, the main theorem of \cite{PetersenWade2024} shows that \[ H^k(\Mod(V_g);\mathbb{Q}) \cong H_{4g-5-k}(\Mod(V_g);D\otimes_\ZZ \QQ), \] where $D$ is the \emph{dualising module} of $\Mod(V_g)$. Furthermore, an explicit description of $D$ is given as the homology of the complex of \emph{non simple disc systems} in $V_g$.

\begin{prob}[Coinvariants of the dualising module]
    Let $D$ denote the dualising module for $\Mod(V)$.
    \begin{enumerate}
        \item Describe a generating set for $D$.  
        \item Is $D$ a cyclic $\QQ\Mod(V)$-module?
        \item Does the rational cohomology of $\Mod(V_g)$ vanish in degree $4g-5$? 
    \end{enumerate}
\end{prob}

Note that for this last point $H^{4g-5}(\Mod(V);\QQ)$ is isomorphic to the coinvariants $(D\otimes_\ZZ \QQ)_{\Mod(V)}$.  For the mapping class group, Harer \cite{Harer1986} showed that the dualising module of $\Mod(S_g)$ is given by the reduced homology of the curve complex. Broaddus then answered the analogous problems to parts (1) and (2) of the above in \cite{Broaddus2012} by giving an explicit description of a nontrivial sphere in $\mathcal{C}(S_g)$ which generates the homology under the orbit of the mapping class group. Note that as the disc complex is a contractible subspace of the curve complex, one has no hope of directly importing Broaddus' sphere to this setting. Church, Farb and Putman then showed how to use Broaddus' work to prove (3) in the affirmative for $\Mod(S_g)$ \cite{ChurchFarbPutman2012}.

Another feature that is known for the mapping class group but not in the handlebody case is that the curve complex is \emph{homotopy equivalent} to a wedge of spheres. In \cite{PetersenWade2024}, the complex of non simple disc systems $\NS(V)$ was only shown to be \emph{homology equivalent} to a wedge of spheres (of dimension $2g-3$ when $b=p=0$ and dimension $2g-4+b+p$ when $b+p >0$). 

\begin{prob}[Connectivity of the non simple disc systems]
Is $\NS(V_g^b)$ simply connected when either $g \geq 3$ or $b> 0$ and $2g - 4 +b\geq 2$ ?  
\end{prob}

The complex $\NS(V_g^b)$ is simply connected when $g \geq 4$ as in this case $\NS(V_g^b)$ contains the entire 2-skeleton of $\mathcal{D}(V_g^b)$, which is contractible. However this problem is particularly delicate when $g=1,2$ (and $b$ is arbitrary), as in this situation many of the 2-cells in the full disc complex come from simple systems.

\subsection{Injectivity of a very natural map}

Using Giansiracusa's work \cite{Giansiracusa2011}, Borinsky, Br\"uck, and Willwacher showed that there exists a subspace \[\text{gr}_0H^*(\Mod(V_g)) \subseteq H^*(\Mod(V_g))\] such that \[ \text{gr}_0H^*(\Mod(V_g)) \cong H^*(\Out(F_g)) \]
(see \cite[Corollary~1.8]{BorinskyBruckWillwacher2025} -- the $\text{gr}_0$ piece of the cohomology comes from a natural grading on an associated graph complex). This abstract isomorphism leads to a very natural problem:

\begin{prob}[Injectivity of the map on cohomology]
Is the map \[\pi^* \colon H^*(\Out(F_g)) \to H^*(\Mod(V_g))\] induced by the projection $\pi \colon \Mod(V_g) \to \Out(F_g)$ injective?
\end{prob}

Let $\text{cv}_g$ be the unparametrised version of Culler-Vogtmann's Outer space. This can be identified with the space of marked, core, metric graphs with fundamental group $F_g$. There is a $\pi$-equivariant map $\Pi \colon \mathcal{H}^\epsilon_g \to \text{cv}_g$ from a handlebody horoball to $\text{cv}_g$ obtained by taking the simple system of meridians of length less than $\epsilon$ to their dual graph, with a curve of length $l$ determining an edge of length $\log(\epsilon/l)$ (the moduli space version of this map appears in Chan--Galatius--Payne \cite{ChanGalatiusPayne2021}). Giansiracusa's theorem proceeds by studying a similar map (see the discussion after Theorem~A in \cite{Giansiracusa2011}), so one might be able to compare the details of its proof with the description appearing in \cite[Corollary~1.8]{BorinskyBruckWillwacher2025}.




\section{Homology growth and \texorpdfstring{$\ell^p$}{lp}-cohomology}

\subsection{Computation of homology growth and \texorpdfstring{$\ell^2$}{l2}-Betti numbers}
Let $\Gamma$ be a residually finite group of type $\mathsf{F}$. L\"{u}ck's celebrated approximation theorem \cite{Lueck1994}, states that if $(\Gamma_k)_{k \in \mathbb{N}}$ is a descending sequence of finite index normal subgroups of $\Gamma$
such that $\bigcap_{k \in \mathbb{N}} \Gamma_k = 1$, then 
\[b_i^{(2)}(\Gamma) = \lim_{k \to \infty}\frac{\mathrm{dim}_{\mathbb{Q}}H_i(\Gamma_k, \mathbb{Q})}{[\Gamma \colon \Gamma_k]}.\]
The quantity on the left hand side being the $i$th $\ell^2$-Betti number of $\Gamma$, a homological invariant which has found many applications in topology and group theory (see \cite{Lueck2002}).  Any sequence of subgroups as above is an example of a \emph{Farber sequence}, see \cite{Farber1998} for a definition.

In analogy we have the \emph{$\FF$-homology growth} for $\FF$ a field and the \emph{homology torsion growth}, which for a Farber sequence $(\Gamma_k)_{k \in \mathbb{N}}$ of $\Gamma$ are defined as
\begin{align*}b^{\FF}_j(\Gamma;(\Gamma_k))&\coloneqq \limsup_{k\to \infty} \frac{\dim_{\FF} H_j (\Gamma_k; \FF)}{[\Gamma : \Gamma_k]} \quad \text{and} \\
t_j^{(2)}(\Gamma;(\Gamma_k))&\coloneqq \limsup_{k\to \infty} \frac{\log | H_j(\Gamma_k;\mathbb Z)_{\rm tors}|}{[\Gamma:\Gamma_k]}\end{align*}
respectively. 

At this point we stress that the definitions above \emph{depend} on the choice of Farber sequence and may not even be limits.  Indeed, the deep part of the statement of L\"uck's approximation theorem is that the limit exists and is independent of the choice of chain.  In the setting above, if the limsup is a genuine limit and there is no dependence of the choice of Farber sequence we denote the invariants by $b^{\FF}_j(\Gamma)$ and $t_j^{(2)}(\Gamma)$ respectively.

A powerful tool for showing the vanishing of the invariants $b^\FF_j(G)$ and $t_j^{(2)}(G)$ is the \emph{cheap rebuilding property} of Abert, Bergeron, Fraczyk, and Gaboriau \cite{AbertBergeronFraczykGaboriau2021}.  Whilst we will not recall the somewhat techincal definition here and insetad refer the reader to \emph{ibid}, we note that the technique has been applied successfully in a number of settings including $\Out(\aster_{i=1}^n\Z/2)$ \cite{GaboriauGuerchHorbez2022}, inner amenable groups \cite{Uschold2024}, and mapping tori of polynomially growing automorphisms \cite{AndrewHughesKudlinska2022,AndrewHughesKudlinskaGuerch2023}.  

\begin{thm}\label{t:rebuilding}
    Let $g\geq 2$.  Then, $\Mod(V_g)$ has the cheap $\alpha$-rebuilding property for all $\alpha\geq 0$. 
 In particular, for each $j\geq 0$ we have $\betti_j(V_g)=0$, and for every Farber sequence $b^\FF_j(V_g)=0$, and $t_j^{(2)}(V_g)=0$. 
\end{thm}
\begin{proof}
By \cite[Theorem~10.10]{AbertBergeronFraczykGaboriau2021} it suffices to establish, for all $\alpha\geq 0$, the cheap $\alpha$-rebuilding property for $\Mod(V_g)$.  To this end we will use the Abert--Bergeron--Fraczyk--Gaboriau combination theorem \cite[Theorem~10.9]{AbertBergeronFraczykGaboriau2021}, which says the following: 

\emph{Let $\Gamma$ be a residually finite group acting on a CW-complex $\Omega$ in such a way that
any element stabilising a cell fixes it pointwise. Let $\alpha \in \NN$. Suppose that:  $\Gamma\backslash\Omega$ has finite $\alpha$-skeleton; $\Omega$ is $(\alpha-1)$-connected;  for each cell $\omega \in \Omega$ of dimension $j \le \alpha$ the stabiliser $\Stab_{\Gamma}(\omega)$ has the cheap $(\alpha-j)$-rebuilding property. 
Then $\Gamma$ itself has the cheap $\alpha$-rebuilding property.}

In our case we will take $X$ to be the disc complex which is a contractible cocompact $\Mod(V_g)$-space.  Another result \cite[Corollary~10.13(1)]{AbertBergeronFraczykGaboriau2021}, states:

\emph{If $H$ is a finite index subgroup of $G$ and $H$ has then cheap $\alpha$-rebuilding property, then so does $G$.} 

Using this result we may pass to a torsion-free finite index subgroup $\Mod_0(V_g)$ of $\Mod(V_g)$ where every mapping class is pure (here we view $\Mod(V_g)$ as a subgroup of $\Mod(S_g)$).  Thus, by the combination theorem we only have to establish the cheap rebuilding property for the stabilisers of $\Mod_0(V_g)$ acting on the disc complex.  

Let $H$ be a cell stabiliser of $\Mod_0(V_g)$ acting on $X$.   Our goal now is to apply \cite[Corollary~10.13(2) and (3)]{AbertBergeronFraczykGaboriau2021} to $H$, which when combined states: 

\emph{If a group $G$ of type $\mathsf{F}_\alpha$ has a normal subgroup $\Z^n$, then $G$ has the cheap $\alpha$-rebuilding property.} 

Now, by \eqref{eqn.stab} we see that $H$ has a finite index subgroup $H_0$ such that $H_0$ has a normal subgroup isomorphic to $\Z^n$ for some $n\geq 1$ and $H_0/\Z^n$ is type $\mathsf{F}_\infty$. It follows from the previous result that $H_0$ has the cheap $\alpha$-rebuilding property for all $\alpha\geq 0$.  Whence the result.
\end{proof}

By the equality $\chi(G)=\sum_{i\geq 0}(-1)^i\betti_i(G)$ we recover the following result of Hirose \cite[Theorem~1.2]{Hirose2003}.  

\begin{corollary}
    $\chi(\Mod(V))=0$.
\end{corollary}

\subsection{\texorpdfstring{$\ell^2$}{l2}-torsion}
A conjecture of L\"uck \cite[Conjecture~1.11(3)]{Lueck2013} predicts that the alternating sum $\sum_{j\geq 0} (-1)^j t_j^{(2)}(G)$ equals the $\ell^2$-torsion $\rho^{(2)}(G)$ of $G$.  See \cite{Lueck2002} for a definition of $\ell^2$-torsion.  Since we have computed that $t_j(\Mod(V))=0$ for all $j$, we raise the following question.
\begin{prob}[$\ell^2$-torsion]
    Is $\rho^{(2)}(\Mod(V))=0$?
\end{prob}

\subsection{\texorpdfstring{$\ell^p$}{lp}-cohomology}
Let $G$ be a group and let $C^\bullet(G)$ denote the singular cochains on $G$ with differential $\delta$.  One can the define $\ell^p$-cohomology $H_{(p)}^\ast(G)$ for any $p\in[1,\infty]$ by considering $\ell^p$-summable cochains $C_{(p)}^\bullet(G)\subseteq C^\bullet(G)$ and then taking cohomology with respect to the differential $\delta_{(p)}=\delta|_{C_{(p)}^\bullet}$.  In general the $\ell^p$-cohomology groups are not Banach spaces, to remedy this one defines the \emph{reduced $\ell^p$-cohomology} $\overline H^\ast_{(p)}(G)$ by taking the quotient by the closure of the image of $\delta$. Theorem~\ref{t:rebuilding} shows that the \emph{reduced} $\ell^2$-cohomology is trivial. 

\begin{prob}[$\ell^p$-cohomology]
    Compute $H^\ast_{(p)}(\Mod(V))$ and $\overline H^\ast_{(p)}(\Mod(V))$.
\end{prob}

\bibliographystyle{halpha}
\bibliography{refs.bib}

\end{document}